\documentclass{amsart}

\usepackage{amssymb,amsmath,amsthm}
\usepackage[all]{xy}
\usepackage{enumitem}
\usepackage{url}

\DeclareMathOperator{\G}{G}

\DeclareMathOperator{\GL}{GL}

\DeclareMathOperator{\jac}{Jac}
\DeclareMathOperator{\trace}{trace}

\theoremstyle{plain}
\newtheorem{theorem}{Theorem}
\newtheorem{lemma}[theorem]{Lemma}
\newtheorem{corollary}[theorem]{Corollary}
\newtheorem{proposition}[theorem]{Proposition}

\theoremstyle{definition}
\newtheorem{definition}[theorem]{Definition}

\newtheorem{remark}[theorem]{Remark}

\newcommand{\Z}{\mathbb{Z}}
\newcommand{\Q}{\mathbb{Q}}
\newcommand{\R}{\mathbb{R}}
\newcommand{\C}{\mathbb{C}}
\newcommand{\F}{\mathbb{F}}
\newcommand{\calO}{\mathcal{O}}
\newcommand{\mm}{\mathfrak{m}}
\newcommand{\ndiv}{\nmid}
\newcommand{\univ}{\mathrm{univ}}
\newcommand{\TT}{\mathbb{T}}

\newcommand{\norm}{\mathrm{Norm}}
\newcommand{\frob}{\mathrm{Frob}}

\title[Level lowering modulo prime powers]{Level lowering modulo prime powers and twisted Fermat
equations}
\author{Sander R. Dahmen}
\address{Department of Mathematics, The University of British Columbia,
Room 121, 1984 Mathematics Road, Vancouver, B.C., Canada V6T 1Z2}
\email{dahmen@math.ubc.ca}
\author{Soroosh Yazdani}
\address{Department of Mathematics \& Statistics, McMaster University, 1280
Main Street, West Hamilton, Ontario, Canada L8S 4K1}
\email{syazdani@math.mcmaster.ca} 
\date{August 31, 2010}

\begin{document}

\begin{abstract}
    We discuss a clean level lowering theorem modulo prime powers
    for weight $2$ cusp forms.
    Furthermore, we illustrate how this can be used to completely
    solve certain twisted Fermat equations
    $ax^n+by^n+cz^n=0$.
\end{abstract}

\maketitle

\section{Introduction}

Since the epoch making proof of Fermat's Last Theorem
\cite{Wiles95}, \cite{TaylorWiles95}, many Diophantine problems
have been resolved using the deep methods developed for FLT and
extensions thereof. One of the basic tools involved are so-called
level lowering results, see e.g. \cite{Ribet90}, \cite{Ribet94}.
These provide congruences between modular forms of different
levels. Until now, all applications of the modular machinery to
Diophantine equations only involved level lowering modulo primes.
Although recently a level lowering result modulo prime powers has
been established \cite{DieulefaitTaixes09}, the statements there are not
very fit for applications to Diophantine equations. The purposes
of this paper are twofold. First of all, we give a clean level
lowering result modulo prime powers that is suitable for
applications to Diophantine equations. Second, we illustrate how
this result can be applied, by completely solving certain twisted
Fermat equations, i.e. Diophantine equations of the form
\[ax^n+by^n+cz^n=0 \quad x,y,z,n \in \Z,\ xyz\not=0,\ n>1\]
where $a,b,c$ are nonzero integers. For the twisted Fermat
equations we consider, the genus one curve defined by
$ax^3+by^3+cz^3=0$ has infinitely many rational points, the curve
defined by $ax^9+by^9+cz^9=0$ has points everywhere locally, and
level lowering modulo $3$ also does not give enough information to
deal with the exponent $n=9$ case. The main application of our
level lowering modulo prime powers theorem is then to use level
lowering modulo $9$ to deal with the exponent $n=9$ case.

The organization of this paper is as follows. In Section
\ref{section level lowering} the level lowering result, Theorem
\ref{thm EllipticLevelLowering}, is stated and proved. In Section
\ref{section irreducibility} we mainly deal with some issues
related to irreducibility of mod $3$ representations. In Section
\ref{section GFE} we solve some twisted Fermat equations using
level lowering modulo primes and level lowering modulo $9$.
Finally, in Section \ref{section other methods} we quickly discuss other
possible methods to attack the twisted Fermat equation for
exponent $n=9$, and we prove that standard level
lowering modulo $3$ methods can never work for our examples.


\section{Level lowering modulo prime powers}\label{section level lowering}
Let $N$ be a positive integer and
$S_2(\Gamma_0(N))$ denote the space of cuspidal modular forms
of weight $2$ with respect to $\Gamma_0(N)$.
For any Hecke eigenform $f \in S_2(\Gamma_0(N))$,
denote by $K_f$ the field of definition of the Fourier coefficients
of $f$, and by $\calO_f$ its ring of integers.
Note that the image of $f$ under different embeddings
of $K_f \rightarrow \C$ gives conjugate Hecke eigenforms in
$S_2(\Gamma_0(N))$. As such, treating $K_f$ as an abstract number
field and $f$ a modular form with Fourier coefficients in $K_f$ is
akin to looking at $f$ and all its Galois conjugates at the same
time. We say $f$ is a {\em newform class}
of level $N$ if $f \in K[ [ q]]$ for a number field $K$,
and the image of $f$ under each (equiv. under any) embedding of $K
\rightarrow \C$ is a normalized Hecke newform in $S_2(\Gamma_0(N))$.
The {\em degree} of the newform class
$f$ is the degree of the number field $K_f$.
Denote by $G_\Q$ the absolute Galois group of $\Q$.
Let $f$ be a newform class of level $N$.
Given a prime $\lambda \subset \calO_f$ lying above $l$, we
can construct (see for example \cite{Shimura94}) a Galois
representation
\begin{equation*}
    \rho^f_{\lambda^r} : G_\Q \rightarrow \GL_2(\calO_f/\lambda^r)
\end{equation*}
for which
\begin{itemize}
    \item $\rho^f_{\lambda^r}$ is unramified away from $Nl$,
    \item $\trace(\rho^f_{\lambda^r}(\frob_p)) \equiv a_p(f) \pmod{\lambda^r}$ and
        $\norm(\rho^f_{\lambda^r}(\frob_p)) \equiv p \pmod{\lambda^r}$ for
        all primes $p \ndiv Nl.$
\end{itemize}
We remark that when $\rho_{\lambda}^f$ is absolutely irreducible, then
$\rho_{\lambda^r}^f$ is uniquely determined (up to change of basis)
for all positive integers $r$ by the congruences above.

Let $E/\Q$ be an elliptic curve of conductor $N$ and minimal
discriminant $\Delta$. Let
\[\rho^E_{l^r}:G_\Q \rightarrow \GL_2(\Z/{l^r}\Z)\]
be the Galois representation coming from the natural Galois action
of $G_\Q$ on $E[l^r](\overline{\Q})$. Assume that $N=N_0 M$ with
$N_0,M \in \Z_{>0}$ and
that there is an odd prime $l$ such that
\begin{itemize}
    \item $N_0$ and $M$ are coprime,
    \item $M$ is square free,
    \item for all primes $p | M$ we have $l|v_p(\Delta)$,
    \item $E[l]$ is irreducible (i.e. $\rho^E_{l}$ is an irreducible Galois
        representation).
\end{itemize}
Then by Ribet's level lowering (\cite{Ribet90}, \cite{Ribet94})
there is a newform class
of level $N_0$ and prime $\lambda \subset \calO_{f}$ lying above $l$
such that
\[ \rho^E_{l} \simeq \rho^f_{\lambda} \]
as Galois representations, or equivalently that $a_p(E) \equiv a_p(f) \pmod
\lambda$ for all primes $p \ndiv N_0l$.

It is natural to ask what happens when $l^r | v_p(\Delta)$ for all
$p | M$ (see \cite{MSRIWorkshop}). The situation in this case is
more complicated. We first need to assume that $E[l]$ is {\em
strongly irreducible} to get around some technical issues with
deformation theory.
\begin{definition}\label{defn strong irreducibility}
    We say a $2$-dimensional Galois representation $\rho$ of $G_\Q$
    is {\em strongly irreducible}, if $\rho|_{G_{\Q(\sqrt{l^*})}}$
    is absolutely irreducible for $l^* = (-1)^{(l-1)/2} l$.
\end{definition}
As noted in \cite{DieulefaitTaixes09}, using results of
\cite{Ribet97}, when $l \geq 5$ and $E$ is
semistable at $l$, then $\rho_l^E$ is strongly irreducible if it
is irreducible. We will deal with the case $l=3$ for elliptic
curves with full rational $2$-torsion in Section
\ref{section irreducibility}.

We also need to assume that there is a unique newform class $f$
and an unramified prime ideal $\lambda$ at level $N_0$ to get the
desired level lowering results.
\begin{theorem} \label{thm EllipticLevelLowering}
    Let $E/\Q$, $N_0$, $M$, $l$ be as above.
    Assume that
    \begin{itemize}
        \item there is a positive integer $r$ such that for all primes
            $p|M$ we have $l^r | v_p(\Delta)$,
        \item for all primes $p | N_0$ we have $l \ndiv v_p(\Delta)$,
        \item $l^2 \ndiv N$,
        \item $E[l]$ is strongly irreducible,
        \item there is a unique pair $(f,\lambda)$ with
            $f$ a newform class of level $N_0$ and $\lambda \subset \calO_f$
            an unramified prime lying above $l$ such that
            $\rho_l^E \simeq \rho_\lambda^f$.
    \end{itemize}
    Then $\rho^E_{l^r} \simeq \rho^f_{\lambda^r}$.
    In particular, if all of the above assumptions are satisfied, then
    \begin{enumerate}[label=(\roman{*}), ref=(\roman{*})]
    \item \label{eqn congruence 1}
        for all primes $p$ with $ p \nmid lN$
        \begin{equation*}
            a_p(f) \equiv a_p(E) \pmod{\lambda^r},
        \end{equation*}
    \item \label{eqn congruence 2}
        for all primes $p$ with $p \nmid N_0 l$ and $p | N$
        \begin{equation*}
            a_p(f) \equiv a_p(E) (1+p) \equiv \pm (1+p) \pmod{\lambda^r}.
        \end{equation*}
    \end{enumerate}
\end{theorem}
\begin{remark}
    Let us explain the reason for the assumptions made in this theorem.
    We need to assume $l^2 \ndiv N$, since the $R=\TT$ results in this
    situation are not strong enough for our applications.
    The assumption that $\lambda$ is unramified is part of the uniqueness, in
    the sense that if $\lambda$ is ramified, then there are two Hecke
    eigenforms $f_1$ and $f_2$ in the same conjugacy class that are congruent
    to each other. Finally, we are assuming that $l \ndiv v_p(\Delta)$ for
    $p|N_0$. We make this assumption since we want to guarantee that
    the Artin conductor of $\rho_l^E$ is $N_0$ and not something smaller.
    This way, we do not have to deal with oldforms for our analysis.
\end{remark}
\begin{remark}
    A similar theorem, using similar techniques, is proved in
    \cite{DieulefaitTaixes09},
    although the statements of the main result there (specialized
    to our case) assume that $M$ is prime, $N$ is square free, and $l
    \ndiv N$.
    Neither of these assumptions are necessary for the main
    proof, and in fact for applications to Diophantine equations these
    assumptions are usually not fulfilled.
\end{remark}

We will present the proof of Theorem \ref{thm EllipticLevelLowering}
for completeness.
The proof uses standard Taylor-Wiles machinery (\cite{Wiles95},
\cite{TaylorWiles95}, see also \cite{Diamond96}) relating the
deformation ring of modular Galois representations to a particular
Hecke algebra. Specifically, let $f$ be a newform class of
level $N_0$, and let $\lambda \subset \calO_f$ be a
prime lying above $l>2$.
Recall that a lifting of $\rho_\lambda^f$ is a
representation
\[ \xymatrix{ \rho : \G_\Q \ar[r] & \GL_2(R) } \]
where $R$ is a Noetherian complete local ring with the maximal ideal $\mm$
and the residue field $R/\mm = \calO_f/\lambda$ such that $\rho
\equiv \rho_\lambda^f \pmod \mm.$ A deformation of $\rho_\lambda^f$ is an
equivalence class of such lifts. We say that $\rho$ is a {\em minimal
deformation} of $\rho_\lambda^f$ if the ramification types of $\rho$ and
$\rho_\lambda^f$ are the same at all primes $p$.

Assume that $\rho_\lambda^f$ is strongly irreducible and semistable at
$l$. Then we know that there is a universal deformation ring
$R^\univ$ and a universal deformation $\rho^\univ: G_\Q
\rightarrow \GL_2(R^\univ)$ such that every minimal deformation
$\rho_\lambda^f$ is strictly equivalent to a unique specialization of
$\rho^\univ$ under a unique homomorphism $R^\univ \rightarrow R$.
Let $\TT$ be the Hecke algebra acting on $S_2(\Gamma_0(N_0))$,
completed at the maximal ideal corresponding to $\rho_\lambda^f$. If we
assume that $N_0$ is the Artin conductor of $\rho_\lambda^f$, then we have
a surjective map $\Phi:R^\univ \rightarrow \TT$. We have
the following celebrated result.
\begin{theorem}[Taylor-Wiles] \label{thm Taylor-Wiles}
Let $l$ be an odd prime. Assume that $\rho^f_{\lambda}$ is
strongly irreducible. Then
\[ \xymatrix{ \Phi:R^\univ \ar[r] &  \TT } \]
is
an isomorphism and $R^\univ$ is a complete intersection.
\end{theorem}
\begin{proof}
    For a proof when $\rho_\lambda^f$ is assumed to be semistable see
    \cite{Wiles95}, \cite{TaylorWiles95}, or \cite{DDT07}.
    To prove the result stated here we also need Diamond's strengthening
    \cite{Diamond96}.
    We remark that in all of the above theorems, the statement
    proved is presented as $R_Q = \TT_Q$ where $R_Q$ is a universal deformation
    ring for certain non-minimal deformations and
    $\TT_Q$ is the
    completed Hecke algebra acting on $S_2(\Gamma_Q)$. The case that
    we are using is when $Q$ is the empty set. In this case
    $R_\emptyset=R$, however the group $\Gamma_\emptyset$
    in the loc. cit. lies between $\Gamma_0(N)$ and $\Gamma_1(N)$.
    Fortunately this group is chosen in such a way that the space
    on which
    the diamond operator is acting trivially modulo $l$, is precisely
    $\Gamma_0(N)$. Therefore $\TT_\emptyset = \TT$.
\end{proof}

As pointed out in \cite{DieulefaitTaixes09}, $R=\TT$ results are
the key to proving level lowering statements.
\begin{proposition} \label{prop LevelLowering}
Let $g$ be a newform class of level $N_g$ and degree $1$.
Assume that there is a pair $(f,\lambda)$
with $f$ a newform class of level $N_f$ and $\lambda \subset \calO_f$ an
unramified prime lying above $l$, and a positive integer $r$
such that
\begin{itemize}
    \item $\rho_l^g$ and $\rho_{l^r}^g$ have Artin conductor $N_f$,
    \item $\rho_\lambda^f \simeq \rho_l^g$,
    \item there is no other pair $(f',\lambda')$ with $f'$ a newform
        class of level $N_f$ and $\lambda' \subset \calO_{f'}$ a
        prime lying above $l$ such that
        $\rho_{\lambda'}^{f'} \simeq \rho_l^g$,
    \item $\rho_l^g$ is strongly irreducible.
\end{itemize}
Then $\rho_{l^r}^g \simeq \rho_{\lambda^r}^f$.
\end{proposition}
\begin{proof}
Let $R^\univ$ be the universal deformation ring for the minimal
deformations of $\rho^f_{\lambda}$. By results of
\cite{TaylorWiles95} we get that $R^\univ = \TT$. Since we are
assuming that $g$ has rational integral coefficients and that
$\rho_\lambda^f \simeq \rho_l^g $,
we get that $\calO_f/\lambda = \Z/l = \F_l$. Since
we are also assuming that there is a unique $(f,\lambda)$, and that
$\lambda$ is unramified, we get $\TT=\calO_{f,\lambda}=\Z_l$.
Furthermore, we have that the Artin conductor of $\rho^g_{l^r}$ is
$N_f$, therefore we get that $\rho^g_{l^r}$ is a minimal
deformation of $\rho^f_{\lambda}$, hence it corresponds to a
map $\TT \rightarrow \Z/l^r$. However, there exists only one
reduction map from $\Z_l$ to $\Z/l^r$, therefore $\rho^g_{l^r}$ is
isomorphic to $\rho^f_{\lambda^r}$.
\end{proof}
\begin{remark}
    In Proposition \ref{prop LevelLowering}, we are
    assuming that $g$ is of degree $1$ to
    simplify the notation and the proof, and because this is the
    case we care most about in this paper.  However, the proof
    does extend to the general case with
    some care.
\end{remark}

We now give the proof of Theorem \ref{thm EllipticLevelLowering}.
\begin{proof}
Let $E/\Q$, $N_0$, $M$, and $l$ be as required. In particular,
assume that $E[l]$ is strongly irreducible. Since we are assuming
that $l^r | v_p(\Delta)$ for all $p| M$ (and $l \ndiv v_p(\Delta)$
for all $l | N_0$) we get that
the Artin conductor of $\rho^E_{l^r}$
is $N_0$. By Ribet's level lowering we get that there is
a newform class $f$ of level $N_0$ and a prime $\lambda$
such that $\rho^f_{\lambda} \simeq \rho^E_{l}$.
Therefore, we can apply Proposition \ref{prop
LevelLowering} to prove that $\rho_{l^r}^E \simeq \rho_{\lambda^r}^f$.
As is well known, the congruences \ref{eqn congruence 1} and
\ref{eqn congruence 2} in the statement of the theorem
follow by comparing the traces of Frobenius.
\end{proof}

\begin{remark}\label{rem levellowering congruence}
    When $f$ is not unique, all hope is not lost and in favourable
    conditions, we can in fact get some explicit level lowering results.
    As an example, consider an elliptic curve $E/\Q$ of conductor
    $71M$ and minimal discriminant $71 M^{27}$, for some square free
    positive integer $M$ coprime to $71$. Furthermore, assume that
    $E[3]$ is strongly irreducible.
    Such an elliptic curve certainly exists,
    for example when $M=2$ we have the elliptic curve
    $142e1$ in the Cremona database
    \[ E: y^2+xy=x^3-x^2-2626x+52244. \]
    By Ribet's level lowering, we can find a newform class
    $f$ of level $71$ and a prime $\lambda \subset \calO_f$ lying above $3$
    such that $\rho_3^E \simeq \rho_\lambda^f$.
    There are two newform classes $f_1$ and $f_2$, each of
    degree $3$, whose complex embeddings span all of $S_2(\Gamma_0(71))$.
    For $i=1,2$, we can check that $3\calO_{f_i}=\lambda_{i,1} \lambda_{i,2}$,
    where
    $\lambda_{i,1}$ is of inertia degree one, while $\lambda_{i,2}$ is of
    inertia degree $2$.
    The image of $\rho_{\lambda_{i,2}}^{f_i}$ is not contained in
    $\GL_2(\F_3)$, therefore $\rho_{\lambda_{i,2}}^{f_i} \not \simeq \rho_3^E$.
    By computing some Fourier coefficients, we get that
    $\rho_{\lambda_{1,1}}^{f_1} \simeq \rho_{\lambda_{2,1}}^{f_2}$.
    We conclude that $\rho_l^E$ is isomorphic to both of these representations.
    Therefore, all the conditions of our level lowering result are fulfilled,
    except for the uniqueness of $(f,\lambda)$. This prevents us from
    proving a level lowering result modulo $27$. However, by studying the
    deformation ring explicitly,
    we can still prove a level lowering result modulo $9$ in the following
    way.
    For $i=1,2,$ we compute that $\calO_{f_i}$ is generated by
    $a_5(f_i)$, explicitly
    \begin{eqnarray*}
        \calO_{f_1} &=&  \Z[t]/<t^3-5t^2-2t+25>,\\
        \calO_{f_2} &=& \Z[t]/<t^3+3t^2-2t-7>.
    \end{eqnarray*}
    Furthermore, the full Hecke algebra acting on $S_2(\Gamma_0(71))$
    has the representation
    \[ \Z[t]/<(t^3-5t^2-2t+25)(t^3+3t^2-2t-7)>\]
    where $t=T_5$ is the fifth Hecke operator.
    Therefore, the universal deformation ring of $\rho_3^E$,
    which is the localization of the Hecke algebra at $\lambda_{i,1}$,
    is
    \[ \TT= \Z_3[t]/<(t-\alpha_1)(t-\alpha_2)>, \]
    where $\alpha_1 \equiv 20 \pmod {27}$ and $\alpha_2 \equiv 11 \pmod {27}$.
    Notice that $\alpha_1 \equiv \alpha_2 \equiv 2 \pmod 9$, which
    means $\rho_{\lambda_{1,1}^2}^{f_1} \simeq \rho_{\lambda_{2,1}^2}^{f_2}$,
    a result that can also be read off from the Fourier coefficients
    of $f_1$ and $f_2$.
    Since $\rho_{27}^E$ is unramified away from
    $3$ and $71$, and it is flat at $3$, we have that
    $\rho_{27}^E$ is a minimal deformation of $\rho_3^E$,
    hence it corresponds to a unique map $R^\univ \rightarrow \Z/27\Z$.
    Note that this gives us two possible maps
    \begin{eqnarray*}
        \psi_i:&  \TT \rightarrow  & \Z_3 \\
         & t \mapsto  & \alpha_i,
    \end{eqnarray*}
    corresponding to the two modular
    forms with coefficients in $\Z_3$.
    Let $\psi: \TT \rightarrow \Z/27\Z$ correspond to $\rho_{27}^E$.
    Note that $\psi$ is
    uniquely defined by the image of $t$, and there are three possible
    choices for this image: $2$, $11$, or $20$.
    Reducing $\psi$ modulo $9$ we get
    $\overline{\psi}: \TT \rightarrow \Z/9\Z$ is given by
    $\overline{\psi}(t)= 2$. Furthermore, $\overline{\psi}$
    corresponds to $\rho_9^E$,
    so we get the following commutative diagram.
    \[ \xymatrix {
        \TT \ar[r]^{\psi_i} \ar[rd]_{\overline{\psi}} & \Z_3 \ar[d] \\
        & \Z/9\Z
    } \]
    By universality, the map
    \[ \xymatrix{ \TT \ar[r]^{\psi_i} & \Z_3 \ar[r] & \Z/9\Z} \]
    corresponds to the reduction of the $\lambda_i$-adic representation
    of $f_i$ modulo $\lambda_i^2$, that is $\rho_{\lambda_i^2}^{f_i}$.
    Since this is the same as $\overline{\psi}$, which corresponds to
    $\rho_9^E$, we get that
    $ \rho_9^E \simeq \rho_{\lambda_i^2}^{f_i}$ for $i=1$ and $2$.

    In case we take for $E$ the elliptic curve $142e1$,
    we can check explicitly that $\rho_{27}^E \not \simeq \rho_{\lambda_{i,1}^3}^{f_i}$
    for $i=1$ or $2$, since $a_5(E)=2$, while
    $a_5(f_i) \equiv \alpha_i \not \equiv 2 \pmod{\lambda_{i,1}^3}$.
    The congruence modulo $9$ can be verified by computing
    some Fourier coefficients explicitly.
\end{remark}

\section{Irreducibility mod $3$}\label{section irreducibility}

In this section, we obtain a criterion for proving that $\rho_3^E$
is strongly irreducible when $E/\Q$ has a full rational $2$-torsion
structure.
We start with a simple lemma.

\begin{lemma}\label{lem Rubin}
Let $E/\Q$ be an elliptic curve. If $\rho_3^E$ is irreducible, but
not strongly irreducible, then $\rho_3^E(G_{\Q})$ is contained in
the normalizer of a split Cartan subgroup of $\GL_2(\F_3)$.
\end{lemma}

\begin{proof}
A (short) proof can be found in \cite[Proposition 6]{Rubin97}.
\end{proof}

Next, a lemma which restricts the possibility of the image of a
mod-$3$ Galois representation attached to an elliptic curve over
$\Q$ with full rational $2$-torsion.

\begin{lemma}\label{lem fiber product}
Let $E/\Q$ be an elliptic curve with full rational $2$-torsion. Then
$\rho_3^E(G_{\Q})$ is not contained in the normalizer of a split
Cartan subgroup of $\GL_2(\F_3)$.
\end{lemma}

\begin{proof}
Consider the modular curves $X_{\mathrm{split}}(3),X(2),X(1)$ and
denote by $j_2$ and $j_{\mathrm{split},3}$ the $j$-maps from
$X(2)$ to $X(1)$ and $X_{\mathrm{split}}(3)$ to $X(1)$
respectively. We have explicitly
\[j_2(s)=2^8\frac{(s^2+s+1)^3}{(s(s+1))^2}\]
and
\[j_{\mathrm{split},3}(t)=12^3\left(\frac{4t+4}{t^2-4}\right)^3.\]
This allows us to explicitly compute the fiber product
$X_{\mathrm{split}}(3) \times_{X(1)} X(2)$ by equating
$j_2(s)=j_{\mathrm{split},3}(t)$, and we let $X$ to be the
desingularization of this fiber product. We compute that $X$ has
genus $1$ and $6$ cusps, all contained in $X(\Q)$. We turn $X$
into an elliptic curve over $\Q$ by taking one of the cusps as the
origin. Now $X$ is isomorphic over $\Q$ to the elliptic curve
determined by
\[y^2 = x^3 - 15x + 22.\]
This curve has rank $0$ and torsion group of order $6$. This shows
that $X(\Q)$ is exactly the set of cusps, which proves the
proposition.
\end{proof}

\begin{corollary}\label{cor irreducibility}
Let $E/\Q$ be an elliptic curve with full rational $2$-torsion. If
$\rho_3^E$ is irreducible, then $\rho_3^E$ is strongly
irreducible.
\end{corollary}

\begin{proof}
This follows immediately by combining Lemma \ref{lem Rubin} and Lemma
\ref{lem fiber product}.
\end{proof}

We still need to a nice criterion for deciding that $\rho_3^E$ is
irreducible.

\begin{lemma}\label{lem irreducibility}
    Let $\mathcal{E}/\F_p$ be an elliptic curve over $\F_p$
    and let $a$ be the trace of Frobenius of this curve.
    Then $\mathcal{E}$ has a $3$-isogeny if and only if $a \equiv \pm(p+1) \pmod 3$.
\end{lemma}
\begin{proof}
    Note that $\mathcal{E}$ has a $3$-isogeny, if and only if
    either $\mathcal{E}$ or its quadratic twist $\mathcal{E}'$
    has a $3$-torsion point, i.e.
    $3 | \#\mathcal{E}(\F_p) = p+1-a$ or $3 | \#\mathcal{E}'(\F_p)=p+1+a$.
    This proves the lemma.
\end{proof}

This brings us to the criterion we need for checking strong
irreducibility.

\begin{proposition}\label{prop irreducibility 3}
    Let $E/\Q$ be an elliptic curve with full rational $2$-torsion and $p\equiv
    1 \pmod{3}$ a prime of good reduction for $E$. If $3|a_p(E)$, then
    $\rho_3^E$ is strongly irreducible.
\end{proposition}
\begin{proof}
    Corollary \ref{cor irreducibility} tells us that irreducibility and
    strong irreducibility in our situation are equivalent. If
    $\rho_3^E$ is reducible, then $E/\Q$ has a rational $3$-isogeny,
    which implies $\overline{E}/\F_p$ has a rational $3$-isogeny for all
    primes of good reduction $p$. By Lemma \ref{lem irreducibility} this
    implies that if $p \equiv 1 \pmod 3$, then we have
    $a_p(E) \neq 0 \pmod 3$, which is the desired result.
\end{proof}

\section{Twisted Fermat equations}\label{section GFE}
Let $a,b,c$ be pairwise coprime nonzero integers and $n>1$ an odd
integer (the case $n$ even is trivial, due to our sign choices).
We are interested in solving the Diophantine equation
\begin{equation} \label{eqn twisted fermat}
    ax^n+by^n+cz^n=0.
\end{equation}
For $n>3$, we know that this equation defines a curve $C_n$ of
genus greater than one, so by Faltings' theorem we get that
$C_n(\Q)$ is finite for any such $n$. In fact, in all the cases we
will consider in this paper, we prove that $C_n(\Q)$ is empty for
all $n>3$, except for one trivial solution when $n=7$ in one of
our examples. For $n>3$ a prime, we will use the modular methods
following \cite{Kraus97} and \cite{HalberstadtKraus02}.
For $n=3$ however, $C_3(\Q) \neq \emptyset$ and the
Jacobian of the curve $C_3$, given by the equation
\begin{equation} \label{eqn jacobian}
    Y^2 = X^3 - 2^4 3^3 (abc)^2,
\end{equation}
is elliptic curve of positive rank in all the cases we are considering.
The only remaining case is when $n=9$, and we note that for our
examples $C_9$ has local points everywhere.
We use our level lowering results modulo prime powers to show
that $C_9(\Q)$ is also empty.

From a Diophantine point of view, the main result is the following.

\begin{theorem}\label{thm main diophantine}
Let $(a,b,c)$ be one of $(5^2,2^4,23^4)$, $(5^8,2^4,37)$,
$(5^7,2^4,59^7)$, $(7,2^4,47^7)$, $(11,2^4,5^2 \cdot 17^2)$. Then for
$n \in \Z_{\geq 2}$ with $n\not=3$ the twisted Fermat equation
(\ref{eqn twisted fermat}) has no solutions $x,y,z$ in integers
with $xyz \neq 0$.
\end{theorem}
\begin{remark}
    Since we are choosing $a$, $b$, and $c$ all positive, proving
    that there are no solutions when $n$ is even is trivial. Of
    course, by different choice of signs, one has to work a little
    bit harder, and we leave those cases to the interested reader.

    In a sense, we have a complete description of the solutions
    for all the exponents $n>1$, since we can find explicit
    generators for the Mordell-Weil group of the
    elliptic curve associated to $C_3$ over $\Q$.
\end{remark}

\subsection{The modular method}
We review how to use elliptic curves, modular forms, and
Galois representations to approach Diophantine equations of the
form (\ref{eqn twisted fermat}), mainly following
\cite{HalberstadtKraus02}. We remark that we deviate from {\em
loc. cit.} by allowing $n$ to be a prime power instead of just a
prime. Fix nonzero coprime integers $a,b,c$. Assume there is a
solution in integers $(x,y,z,n)$ to (\ref{eqn twisted fermat}),
with $xyz \neq 0$ and $n\geq 3$ odd.
Without loss of generality we assume that $by^n$
is even and that $ax^n \equiv -1 \pmod 4$. We also assume
that $ax^n, by^n,$ and $cz^n$ are pairwise coprime.
Consider the Frey elliptic curve
\begin{equation} \label{eqn frey curve}
    Y^2=X(X-ax^n)(X+by^n).
\end{equation}
This model is minimal at all odd primes. Furthermore, if $b \equiv 0 \pmod{16}$,
then we can find a global minimal model over $\Z$
\begin{equation} \label{eqn minimal frey curve}
    E_{n,(x,y)} : Y^2+XY=X^3+{by^n-ax^n-1 \over 4} X^2-{abx^ny^n \over 16} X.
\end{equation}
We often simply write $E_n$, or even $E$, when the indices are
understood from the context. We assume that the condition on $b$
is satisfied, since these are the only cases that we will consider
(for the general situation the reader can refer to
\cite{HalberstadtKraus02}). The minimal discriminant and the
conductor of $E_n$ are given by
\begin{eqnarray*}
    \Delta(E_n) &=& \frac{(abcx^ny^nz^n)^2}{2^8}, \\
    N(E_n) &=& \prod_{\substack{p | abcxyz \\ p \mathrm{\ odd\ prime}}} p.
\end{eqnarray*}
Consider the Galois representation
\begin{equation*}
    \rho_n^{E_n}: G_\Q \rightarrow \GL_2(\Z/n\Z).
\end{equation*}
Assume $n=l^r$ is a prime power with $l$ an odd prime and $r$ a positive
integer, and that $\rho_l^{E_n}$ is strongly irreducible.
By Ribet's level lowering \cite{Ribet90}, \cite{Ribet94}
and the work of Wiles and Taylor-Wiles \cite{Wiles95},
\cite{TaylorWiles95}, we know that $\rho_l^{E_n}$ arises from a
newform class $f$ of level
\[ N_0 = \prod_{\substack{p | abc \\ p>2}} p. \]
This means that there exists a prime ideal $\lambda \subset
\calO_f$ lying above $l$, such that
\begin{equation} \label{eqn isom reps}
    \rho_l^{E_n} \simeq \rho_\lambda^f.
\end{equation}

\begin{remark}
    Note that, since we are assuming that $a, b,$ and $c$ are pairwise
    coprime to each other, if $p^n \ndiv abc$ with $p$ prime,
    then for showing that
    (\ref{eqn twisted fermat}) has no nontrivial solutions, we can
    assume without loss of generality that $ax^n, by^n,$ and $cz^n$
    are pairwise coprime. This is the case for all of our examples
    when $n>7$.
    In general, if for some prime $p$, we have that $p$ divides
    $ax^n, by^n$ and $cz^n$, then it is possible that $E_n$
    has additive reduction at $p$, however for odd primes $p$,
    a quadratic twist of $E_n$ will have a semistable reduction at $p$.
    We will replace $E_n$ by its appropriate quadratic twist (if necessary)
    for the rest of this paper.
    All of our computations in this case will be the same, except
    that we might end up at a level dividing $N_0$, however if we
    use the level $N_0$, the argument still goes through.
    (The situation at the prime $2$ is a bit more subtle, however we
    do not have to deal with it when $n \geq 5$ in our examples.)
    Note that if we use the smaller level, the argument can be
    easier. For instance, in the example
    \[ 5^7x^n+2^4y^n+59^7z^n=0 \]
    when $n=7$, we only need to consider newform classes of level $1$
    (of which there are none). This gives us a considerably easier
    contradiction. However, for sake of uniformity,
    we actually deal with newform classes of level $295$ to show
    that this equation has no nontrivial solutions.
\end{remark}
For every equation in Theorem \ref{thm main
diophantine}, the level $N_0$ we need to consider is given in
Table \ref{table primes}. By comparing traces of Frobenius, we
obtain the congruences \ref{eqn congruence 1} and \ref{eqn
congruence 2} with $r=1$. If $(f,\lambda)$ is the unique pair of
newform class $f$ of level $N_0$ and prime $\lambda \subset
\calO_f$ lying above $l$ that satisfies (\ref{eqn isom reps}),
$\lambda$ is unramified, and for all primes $p | N_0$ we have
$l \ndiv v_p(\Delta(E_n))$,
then we can apply Theorem \ref{thm EllipticLevelLowering} to
$\rho_n^{E_n}$. In this case we get that
$\rho_{n}^{E_n}\simeq \rho_{\lambda^r}^f$ and in particular that the
congruences \ref{eqn congruence 1} and \ref{eqn congruence 2}
hold.

\subsubsection{$l \geq 5$}
Let $l \geq 5$ be a prime number.
Note that by the arguments of \cite[Proposition
6]{Serre87} we have $\rho_l^{E_l}$ is irreducible,
since $E$ is semistable and has full rational $2$-torsion
(by our earlier remarks, this tells us that $\rho_l^{E_l}$ is
strongly irreducible).
In order to prove that there are no solutions to (\ref{eqn twisted
fermat}) for $n=l$, it suffices to find a
contradiction (using congruences) to (\ref{eqn isom reps}) for all
pairs $(f,\lambda)$ of newform classes $f$ of level $N_0$ and
primes $\lambda \subset \calO_f$ lying above $l$.

Let $f$ be a newform class of level $N_0$ such that (\ref{eqn
isom reps}) holds for some $\lambda$. For any prime $p$, define
\[ \mathcal{A}_p= \begin{cases}
\{ a \in \Z : a \equiv p+1 \pmod{4} \mathrm{\ and\ } |a| \leq 2 \sqrt{p} \} \quad \mbox{if $p$ is odd};\\
\{-1,1\} \quad \mbox{if $p=2$}.
\end{cases}\]
We claim that for all primes $p$ where $E$ has good reduction we
have $a_p(E) \in \mathcal{A}_p$. This is because $E$ has full rational
$2$-torsion and for an odd prime $p$ of good reduction, $E[2]$
injects into the reduction of $E$ modulo $p$. If $E$ has good
reduction at $p=2$, then one checks that the reduction of $E$
modulo $p$ still has a rational $2$-torsion point. Together with the Weil
bound, the claim follows. Next, define for all primes $p$ the set
\[\mathcal{T}_p=A_p \cup \{\pm(p+1)\}.\]
The congruences \ref{eqn congruence 1},\ref{eqn congruence 2}
with $r=1$ now give us that for a prime $p \nmid N_0$ we have
\[l| L_{f,p}=p\prod_{a \in \mathcal{T}_p} \norm(a-a_p(f)).\]
(If the degree of $f$ is equal to $1$, then the prime $p$ before
the product is not necessary, but in all our examples this does
not lead to any new information.) It is of course possible that
$L_{f,p}=0$, in which case $l|L_{f,p}$ does not give any
information. However, all our examples are chosen such that,
either $f$ is not rational, or it is rational and the elliptic
curve of conductor $N_0$ associated to it by the Eichler-Shimura
relation is not isogenous to an elliptic curve with full rational
$2$-torsion. In what follows, assume that $f$ satisfies these
conditions. This implies that for infinitely many (in fact, a
positive proportion of) primes $p$ we have $a_p(f) \not\in
\mathcal{A}_p$ (and hence $a_p(f) \not\in \mathcal{T}_p$ since by
the Weil bounds $a_p(f)\not=\pm (p+1)$). It is easy to get an
upper bound in terms of $N_0$ (or $a,b,c$) for the smallest prime
$p \nmid N_0$ for which $a_p(f)\not\in \mathcal{A}_p$. In
practice, for several primes $p_{\mathrm{max}}$, we compute
\[ \gcd \left\{L_{f,p}: \mathrm{\ primes\ } p\leq p_{\mathrm{max}} \mathrm{\ with\ } p \nmid N_0 \right\}\]
and the set of odd primes dividing this quantity, denoted
$\mathcal{L}_{f,p_{\mathrm{max}}}$. We do this until we find a
prime $p_{\mathrm{max}}$ for which
$\mathcal{L}_{f,p_{\mathrm{max}}}$ is not empty and appears to be
the same as $\mathcal{L}_{f,p_{\mathrm{max}}'}$ for any prime
$p_{\mathrm{max}}'\geq p_{\mathrm{max}}$. This yields the
information that for all primes $l>3$ such that $l \not \in
\mathcal{L}_{f,p_{\mathrm{max}}}$ we have a contradiction to
(\ref{eqn isom reps}) for all primes $\lambda \subset \calO_f$
lying above $l$.

For the given level $N_0$, let us finally define for a prime
$p_{\mathrm{max}}$ the set
\[\mathcal{L}_{p_{\mathrm{max}}}=\bigcup_f \mathcal{L}_{f,p_{\mathrm{max}}}\]
where the union is over all newform classes $f$ of level $N_0$. By taking
$p_{\mathrm{max}}$ to be the maximum of the $p_{\mathrm{max}}$'s
for the newform classes $f$, we arrive at a finite set of odd primes
$\mathcal{L}_{p_{\mathrm{max}}}$ which contains $3$ and, in
practice, just a few other odd primes. For every equation in
Theorem \ref{thm main diophantine}, a value of $p_{\mathrm{max}}$
together with $\mathcal{L}_{p_{\mathrm{max}}}-\{3\}$ is given in
Table \ref{table primes}. The significance for the original
Diophantine problem is that for every odd prime $l \not\in
\mathcal{L}_{p_{\mathrm{max}}}$ we have
that (\ref{eqn twisted fermat}) has no integer solutions
with $xyz \neq 0$.

In our examples, for the finitely many primes $l\geq 5$ contained
in $\mathcal{L}_{p_{\mathrm{max}}}$, we show that
$C_l(\Q)=\emptyset$ (except for the one trivial exception)
either by finding a prime $p$ for which
$C_n(\Q_p)=\emptyset$ or, if no such prime $p$ exists, by using
Kraus' method of reduction, see \cite{Kraus98} or \cite[Section
1.2.]{HalberstadtKraus02}, which we briefly describe now.

Fix an exponent $n=l^r$ a prime power (in {\em loc. cit.} $n$
is assumed to be a prime). The possibilities for $a_p(E_n)$ (and
$\trace(\rho_l^{E_n}(\frob_p))$) with $p\equiv 1 \pmod{l}$
sometimes can be shown to be strictly smaller than $\mathcal{A}_p,$ (and
$\mathcal{T}_p$ respectively) by using the additional information
that (\ref{eqn twisted fermat}) has to be satisfied modulo $p$.
Let $p \nmid lN_0$ be a prime. For an element $q \in \Q$ whose
denominator is not divisible by $p$, we denote by $\overline{q}$
the reduction of $q$ modulo $p$ in $\F_p$. If (\ref{eqn twisted
fermat}) has an integer solution $(x,y,z)$ with $p|xyz$ other than
$(0,0,0)$, then necessarily
\begin{equation}\label{eqn 2nd case sols}
\overline{a/b}, \overline{b/c}, \mathrm{\ or\ } \overline{c/a} \in
{\F_p^*}^n.
\end{equation}
In this case we get from $\rho_l^{E_{n,(x,y)}} \simeq \rho_l^f$
that $a_p(f) \equiv \pm (p+1) \pmod{\lambda}$.
If the (hypothetical) integer solution $(x,y,z)$ to (\ref{eqn
twisted fermat}) satisfies $p \nmid xyz$, then
$(\overline{x},\overline{y})$ belongs to
\[S_{n,p}=\{(\alpha,\beta) \in \F_p^* \times \F_p^* : \overline{a}\alpha^n+\overline{b}\beta^n+\overline{c}\gamma^n=0 \ \mbox{for some $\gamma \in \F_p^*$}\}.\]
For any $P=(\alpha,\beta) \in S_{n,p}$, define an elliptic curve
over $\F_p$ by
\[\mathcal{E}_{n,p,P}:y^2=x(x-\overline{a} \alpha^n)(x+\overline{b} \beta^n).\]
Then $a_p(E_{n,(x,y)})$ belongs to
\[\mathcal{A}_{n,p}=\{a_p(\mathcal{E}_{n,p,P}) : P \in S_{n,p}\}.\]
Also consider the set (of possibilities for
$\trace(\rho_l^{E_n}(\frob_p))$)
\[\mathcal{T}_{n,p}=\begin{cases}
\mathcal{A}_{n,p} \cup \{\pm (p+1)\} \quad \mbox{if (\ref{eqn 2nd case sols}) holds};\\
\mathcal{A}_{n,p} \quad \mbox{otherwise}.
\end{cases}\]
Hence, in order to prove that for a (hypothetical) solution
$(x,y,z,n)$ to (\ref{eqn twisted fermat}) and a certain newform class
$f$ of level $N_0$ we cannot have $\rho_l^{E_{n,(x,y)}} \simeq
\rho_{\lambda}^f$ for any prime $\lambda \subset \calO_f$ lying above
$l$, it suffices to find a prime $p \nmid N_0l$ such that
\[l \nmid \prod_{a \in \mathcal{T}_{n,p}} \norm(a-a_p(f)).\]
If for all newform classes $f$ of level $N_0$ we can find such a prime
$p$, then we conclude that (\ref{eqn twisted fermat}) has no integer
solutions with $xyz \neq 0$. In practice, since
we already computed $\mathcal{L}_{f,p_\mathrm{max}}$ for some \lq
large\rq\ prime $p_{\mathrm{max}}$, it only remains to find such a
prime $p$ for the newform classes $f$ of level $N_0$ for which $l \in
\mathcal{L}_{f,p_\mathrm{max}}$.

\begin{remark}\label{rem twist}
From a computational point of view, it is worthwhile to only
consider $\mathcal{E}_{n,p,P}$ up to quadratic twist in order to determine
$\mathcal{A}_{n,p}$ (and hence $\mathcal{T}_{n,p}$). To be
specific, let
\[S_{n,p}'=\{\alpha \in \F_p^* : \overline{a/c}\alpha^n+\overline{b/c} \in {\F_p^*}^n \}.\]
Then we get
\[\mathcal{A}_{n,p}=\{\pm a_p(\mathcal{E}_{n,p,(\alpha,1)}) : \alpha \in S_{k,p}'\}.\]
\end{remark}

For every equation in Theorem \ref{thm main diophantine}, an entry
$(l,p)$ under \lq local $(l,p)$\rq\ of Table \ref{table primes}
indicates that $C_l(\Q_p)=\emptyset$. Furthermore,
for every prime $l \geq 5$ and newform class $f$ of level $N_0$
for which $l \in \mathcal{L}_{f,p_{\mathrm{max}}}$
(which implies $l \in \mathcal{L}_{p_{\mathrm{max}}}$) and
$C_l$ is locally solvable everywhere,
there is an entry $(l,p)$ in
Table \ref{table primes} under \lq Kraus $(l,p)$\rq\
indicating that $l \nmid
\prod_{a \in \mathcal{T}_{l,p}} \norm(a-a_p(f))$. This completes
the data that proves Theorem \ref{thm main diophantine} for primes
$l \geq 5$.

\begin{table}[th!]
\begin{center}
\begin{tabular}{c|c|c|c|c|c}
$(a,b,c)$ & level & $p_{\mathrm{max}}$ &
$\mathcal{L}_{p_{\mathrm{max}}}-\{3\}$ & local $(l,p)$ & Kraus
$(l,p)$
\\ \hline
$(5^2,2^4,23^4)$ & $115$ & $3$ & $\{5\}$ & $(5,11)$ & - \\
$(5^8,2^4,37)$ & $185$ & $3$ & $\{5,19\}$ & $(19,19)$ & $(5,31)$ \\
$(5^7,2^4,59^7)$ & $295$ & $3$ & $\{5,7\}$ & $(5,5)$ & $(7,43)$ \\
$(7,2^4,47^7)$ & $329$ & $23$ & $\{5\}$ & - & $(5,11)$ and $(5,41)$ \\
$(11,2^4,5^2 \cdot 17^2)$ & $935$ & $71$ & $\{5,7\}$ & $(5,5)$ &
$(7,29)$
\end{tabular}
\caption{Data for primes $l\geq 5$.}\label{table primes}
\end{center}
\end{table}

\subsubsection{$n=9$}

To prove that $C_9(\Q)=\emptyset$ for our curves, we first note
that $C_9$ has points everywhere locally; this is a
straightforward computation. We can also quickly find a rational
point on $C_3$ and find that the corresponding elliptic curve has
positive rank over $\Q$; see Table \ref{table 3,9}. Next, we start
applying a mod-$3$ modular approach. Before we can apply level
lowering, we need to show that $\rho_3^{E_9}$ is irreducible
(which implies that it is absolutely irreducible). Later on, we
shall need the stronger result that $\rho_3^{E_9}$ is strongly
irreducible. Using Proposition \ref{prop irreducibility 3} we
obtain an explicit criterion for checking this.

\begin{proposition}\label{prop explicit irreducibility}
Suppose $p$ is a prime such that
\begin{itemize}
\item $p \nmid abc$,
\item $p\equiv 1 \pmod{9}$,
\item condition (\ref{eqn 2nd case sols}) does \emph{not} hold,
\item for all $P \in S_{9,p}$ we have $3|\#\mathcal{E}_{9,p,P}$.
\end{itemize}
Then $\rho_3^{E_9}$ is strongly irreducible.
\end{proposition}

Notice that, as in Remark \ref{rem twist}, by considering
quadratic twists, the last condition in the proposition above can
be replaced by: for all $\alpha \in S_{9,p}'$ we have
$3|\#\mathcal{E}_{9,p,(\alpha,1)}$. In all our examples, we find a prime
$p=p_{\mathrm{irr}}$ satisfying all the conditions in the
proposition above, the value is recorded in Table \ref{table 3,9}.

Although we would like to apply Theorem \ref{thm
EllipticLevelLowering} with $r=2$ and $l=3$, for most of the
levels $N_0$ we are considering there actually exist distinct
pairs $(f_1,\lambda_1),(f_2,\lambda_2)$ of newform classes
$f_1,f_2$ of level $N_0$ and prime ideals $\lambda_i \subset
\calO_{f_i}$ ($i=1,2$) lying above $3$ for which
$\rho_{f_1}^{\lambda_1} \simeq \rho_{f_2}^{\lambda_2}$ holds. So
we start with applying \lq ordinary\rq\ level lowering mod-$3$.
For every newform class $f$ of level $N_0$ with $3 \in
\mathcal{L}_{f,p_{\mathrm{max}}}$ we compute for various primes
$p\nmid abc$ with $p\equiv 1 \pmod{3}$ the set $\mathcal{T}_{9,p}$
and check if
\begin{equation} \label{eqn kraus 3 n=9}
    3 \nmid \prod_{a \in \mathcal{T}_{9,p}} \norm(a-a_p(f)).
\end{equation}
Denote by $\mathcal{N}_{p_0}$ the set of newform classes $f$ of
level $N_0$
such that for all primes $p \leq p_0$, (\ref{eqn kraus 3 n=9}) does
not hold.  The prime $p_0$ we used, together with a description of
$\mathcal{N}_{p_0}$ is given in
Table \ref{table 3,9}. Now if $\rho_3^{E_9} \simeq \rho_\lambda^f$
for some newform class $f$ of level $N_0$ and prime ideal $\lambda
\subset \calO_f$ lying above $3$, then necessarily $f \in
\mathcal{N}_{p_0}$. For all the cases we consider, taking into
account that the image of the representation has to be contained
in $\GL_2(\F_3)$, we find that $\lambda$ has to be of inertia
degree $1$. Furthermore, we find a $p_0$ such that
$\mathcal{N}_{p_0}$ is small enough so that the uniqueness
condition in Theorem \ref{thm EllipticLevelLowering} now holds. As
for the ramification condition, in all our cases, except when
$N_0=935$, we find that $3$ is unramified in $K_f$ for all $f \in
\mathcal{N}_{p_0}$. For the case where $N_0=935$, we actually find
that the pair $(f,\lambda)$ we can not deal with modulo $3$ has
$\lambda$ ramified in $K_f$,
and we deal with this case by studying the deformation ring directly.
In all other cases, we apply Theorem \ref{thm EllipticLevelLowering}
(with $r=2$ and $l=3$)
to all newform classes in $\mathcal{N}_{p_0}$.
This time, for all $f \in \mathcal{N}_{p_0}$,
we simply try to find a prime $p\nmid 3N_0$
such that for all $a \in \mathcal{T}_p$ we have
\[9 \nmid \norm(a-a_p(f))\]
and $\Q(a_p(f))=K_f$. In all cases we readily do find such a prime
$p$; for the value, see Table \ref{table 3,9}. This means that we
also obtain a contradiction from $\rho_9^{E_9} \simeq
\rho_{\lambda^2}^f$ for all $f \in \mathcal{N}_{p_0}$ and all
relevant $\lambda$. It follows that $C_9(\Q)=\emptyset$.

\begin{table}[ht!]
\begin{center}
\begin{tabular}{c|c|c|c|c|c}
level & $\Q$-rank of $C_3$ & $p_{\mathrm{irr}}$ & $p_0$ & $\mathcal{N}_{p_0}$ description & $p$ \\
\hline
$115$ & $2$ & $73$ & - & $d=1$ & 2 \\
$185$ & $1$ & $73$ & $73$ & $d=1^*$ & 2 \\
$295$ & $2$ & $37$ & $37$ & $d=6$ & 13 \\
$329$ & $2$ & $109$ & $13$ & $d=5,d=6$ & 5,5 \\
$935$ & $2$ & $37$ & $37$ & $d=11^*$ & -
\end{tabular}
\caption{Data for $n=3$ and $n=9$; $d$ denotes the degree of the
newform class. In case of a $*$, the degree alone does not
determine the class uniquely, in which case we impose the extra
condition $\trace(a_2(f))=0$, which does determine the class
uniquely. }\label{table 3,9}
\end{center}
\end{table}

\begin{remark}\label{remark kraus norm}
In other cases, using $\mathcal{T}_p$ might not give enough
information. Using $\mathcal{T}_{9,p}$ instead might lead to the
desired conclusion. Furthermore, computing $a_p(f) \pmod \lambda$ may yield
more information than $\norm(a-a_p(f)) \pmod l$.
For example, if $(a,b,c)=(7,2^4,47^7)$,
computing $\norm(a-a_p(f)) \pmod 3$ did not give us enough
information to rule out the newform classes of level $329$ with
degrees $5$ and $6$. However, the newform class $f$ of degree $6$
can be ruled out using $a_p(f) \pmod \lambda$.
Specifically, there is a unique prime $\lambda$ above $3$ of
inertia degree $1$ in $K_f$. For any other prime $\lambda'$ above
$3$, the image of the Galois representation $\rho_{\lambda'}^f$ is
not contained in $\GL_2(\F_3)$, hence not isomorphic to
$\rho_3^{E_9}$. To rule out the
prime $\lambda$ in this case, we just note that $f \mod \lambda$ is an
Eisenstein series, which means $\rho_\lambda^f$ is  reducible,
hence not isomorphic to $\rho_3^{E_9}$ since we proved this is
irreducible.

In fact, using these observations, in all our cases we find that
there is exactly one pair of $(f,\lambda)$ for which we are unable
to obtain a contradiction using just level lowering modulo $3$.
In Section \ref{section other methods} we will prove that level
lowering modulo $3$ is not sufficient for this pair.
\end{remark}

\subsubsection{Level $935$}\label{section 935}
We now consider the example
\[ 11x^9+2^4 y^9+5^2 \cdot 17^2 z^9 = 0.\]
Using the data in Table \ref{table 3,9}, we conclude that
if there is a nontrivial solution to this equation,
then we have
$\rho_3^{E_9} \not \simeq \rho^{f_{11}}_{\lambda}$,
where $f_{11}$ is the newform class of level $935$ with degree
$11$ and $\trace(a_2(f_{11}))=0$, and $\lambda \subset \calO_{f_{11}}$
a prime lying above $3$ of inertia degree $1$.
There are exactly two primes $\lambda_1, \lambda_2\subset
\calO_{f_{11}}$ lying above $3$ of inertia degree $1$.
One is unramified, say $\lambda_1$, but the other, say $\lambda_2$
is ramified.
As for the representations
$\rho_{\lambda_i}^{f_{11}}$ for $i=1,2$, we quickly find that they
are not isomorphic.

The case $\rho_3^{E_9} \simeq \rho_{\lambda_1}^{f_{11}}$ actually
leads to a contradiction fairly easily by computing
$\mathcal{T}_{9,31}=\left\{ \pm 8, \pm 32 \right\}$ and
$a_{31}(f_{11}) \equiv 0 \pmod {\lambda_1}$.

The situation for $(f_{11},\lambda_2)$ is a bit more delicate,
specially since we can not apply our level lowering theorem in
this situation. However we can still rule this case out by using
the deformation ring directly. Assume that
$\rho_3^{E_9} \simeq \rho_{\lambda_2}^{f_{11}}$. Then
$\rho_9^{E_9}$ is a minimal deformation of
$\rho_{\lambda_2}^{f_{11}}$. Therefore, this representation should
correspond to a unique map $\R^\univ \rightarrow \Z/9\Z$. However
we can explicitly compute $\R^\univ = \TT$, as we did in remark
\ref{rem levellowering congruence}. Since $(f_{11},\lambda_2)$ is
not congruent to any other newform class of level $935$, we get
that $\TT=\left(\calO_{f_{11}}\right)_{\lambda_2}$. Using SAGE or
MAGMA we get that $\Z[a_3(f_{11})] \subset \calO_{f_{11}}$ with an
index coprime to $3$. Therefore
\[ \calO_{f_{11}} \otimes \Z_3 = \Z_3[T]/<P(T)> \]
where
\begin{small}
\[
P(T) = T^{11} - T^{10} - 25 T^{9} + 26 T^{8} + 222 T^{7} - 225 T^{6} - 827 T^{5} + 705 T^{4} + 1212 T^{3} - 449 T^{2} - 770 T - 168 \]
\end{small}
is the minimal polynomial for $a_3(f_{11})$.
Factoring $P(T)$ over $\Z_3$, we can conclude that
\[ \TT=\left(\calO_{f_{11}}\right)_{\lambda_2} = \Z_3[T]/<T^2-aT+b> \]
where $a \equiv 4 \pmod 9$ and $b \equiv 7 \pmod 9$. Looking at
the above equation modulo $9$, we get $T^2-aT+b \equiv (x-2)^2+3
\pmod 9$, which implies that there are no maps $\TT \rightarrow
\Z/9\Z$. This gives us a contradiction to our assumption that
$\rho_9^{E_9}$ is a minimal deformation of
$\rho_{\lambda_2}^{f_{11}}$, which rules this modular form out as
well, and proves our result.

\section{Necessity of level lowering modulo $9$}\label{section other methods}
One can ask if we really needed level lowering modulo $9$ for solving
(\ref{eqn twisted fermat}).
A priori it could be possible that by using level lowering modulo $3$
we can already obtain the desired contradictions.
However, we will show that for our choice of Frey curve and
triples $(a,b,c)$, level lowering modulo $3$ does not yield
enough information.

\begin{remark}
    We note that there are Frey curves
    attached to the Diophantine equations
    \begin{eqnarray*}
        ax^n+by^n+cz^3&=& 0 \\
        ax^3+by^3+cz^n&=& 0,
    \end{eqnarray*}
    and we can specialize these curves to the case $n=9$.
    The Frey curve attached to the first equation has
    a rational $3$-isogeny by construction, so it is not suitable
    for level lowering.
    As for the Frey curves attached to the second equation,
    along similar lines as in Section \ref{section irreducibility},
    one can show that $E[3]$ is strongly irreducible.
    However, we end up having to deal with modular forms of higher level,
    for example when
    $(a,b,c)=(11,2^4,5^2\cdot 17^2)$, we have to deal with
    modular forms of level (at least) $92565$,
    which is computationally very difficult.

    We also note that other possible non-modular approaches to proving
    $C_9(\Q)=\emptyset$ include descent methods and Mordell-Weil sieving.
    This is a promising approach, almost completely orthogonal
    to the modular method presented here and
    can be an interesting topic for further investigation.
\end{remark}

Let $\overline{\mathcal{T}_{n,p}}\subset \F_3$ be the image of
$\mathcal{T}_{3,p}$ under the reduction map modulo $3$.
In this section we will show that for our examples,
the unique pair of newform class $f$ of level $N_0$ and prime
$\lambda \subset \calO_f$ lying above $3$
mentioned at the end of Remark \ref{remark kraus norm}
satisfies
\[ a_p(f) \mod \lambda \in \overline{\mathcal{T}_{9,p}}\]
for all primes $p \ndiv 3N_0$.
This means that level lowering modulo $9$ (along with the
argument in Section \ref{section 935}) is necessary to prove
Theorem \ref{thm main diophantine}.

We will start by showing that
$\overline{\mathcal{T}_{3,p}} \neq \F_3$ infinitely
often, by showing $0 \not \in \overline{\mathcal{T}_{3,p}}$ for
infinitely many $p \equiv 1 \pmod 3$.
\begin{lemma} \label{lem t3p}
    Let \[ \xymatrix{j : C_n \ar[r] & X(2) \ar[r] & X(1)} \]
    be the $j$-invariant of the Frey elliptic curve corresponding
    to a point on $C_n$, let $p \equiv 1 \pmod n$ be a prime with
    $p \ndiv N_0n$, and let
    $\pi : X_0(3) \rightarrow X(1)$ be the natural forgetful map
    between the modular curves.
    Then
    \begin{enumerate}
        \item $0 \not \in \overline{\mathcal{T}_{n,p}}$ if and only if
            $j(C_n(\F_p)) \subset \pi(X_0(3)(\F_p))$,
            if and only if
            \[\left(C_n \times_{X(1)} X_0(3)\right) (\F_p) \rightarrow C_n(\F_p) \]
            is surjective.
        \item $\pm 1 \in \overline{\mathcal{T}_{n,p}}$ if and only
            if there is a $z \in C_n(\F_p)$ such that
            $j(z) \in \pi(X_0(3)(\F_p))$,
            if and only if $\left(C_n \times_{X(1)} X_0(3)\right)(\F_p)$
            is not empty.
    \end{enumerate}
\end{lemma}
\begin{proof}
    If $0 \not \in \overline{\mathcal{T}_{n,p}}$, then
    by Lemma \ref{lem irreducibility} we get that for all $z \in j(C_n(\F_p))$
    such that $z \neq \infty$ we have that the corresponding elliptic
    curve $\mathcal{E}_n/\F_p$ has a $3$-isogeny, i.e. $z \in \pi(X_0(3)(\F_p))$.
    We also know that $\infty \in \pi(X_0(3)(\F_p))$, therefore we
    get $j(C_n(\F_p)) \subset \pi(X_0(3)(\F_p))$.
    Now assume that $j(C_n(\F_p)) \subset \pi(X_0(3)(\F_p))$.
    Let $z \in j(C_n(\F_p))$. Notice that if $z=\infty$ then
    the trace of Frobenius of the corresponding generalized elliptic
    curve is $\pm(p+1) \not \equiv 0 \pmod 3$. If
    $z \neq \infty$, then by our assumption $z$ corresponds to
    an elliptic curve $\mathcal{E}_n/\F_p$ with a rational three
    isogeny, which by Lemma \ref{lem irreducibility} will have
    trace of Frobenius equivalent to $\pm 1$ modulo $3$. In either case
    we get $0 \not \in \overline{\mathcal{T}_{n,p}}$.
    This proves the first equivalence.
    By definition of fiber products we see that
    $j(C_n(\F_p)) \subset \pi(X_0(3)(\F_p))$ if and only if
    $\left(C_n \times_{X(1)} X_0(3)\right)(\F_p) \rightarrow C_n(\F_p)$ is surjective,
    which finishes the first part of the lemma.

    The second part of the lemma is proved in a similar fashion.
\end{proof}
Therefore, to find primes $p$ such that $0 \not \in
\overline{\mathcal{T}_{3,p}}$, we are reduced to finding $p$ such
that $\left(C_3 \times_{X(1)} X_0(3)\right)(\F_p) \rightarrow C_3(\F_p)$ is not
surjective. For brevity, we will drop the $X(1)$ from the fiber
product. Let $\widetilde{C_3 \times X_0(3)}$ be the
desingularization of the fiber product. The following lemma
describes the map $\widetilde{C_3 \times X_0(3)} \rightarrow C_3$.
\begin{lemma} \label{lem genus}
    The curve $C_3 \times X_0(3)$ is a genus one curve.
    Furthermore, the natural map
    $\widetilde{C_3 \times X_0(3)} \rightarrow C_3$
    induces the multiplication by $2$ on their Jacobians.
    In particular the Jacobian of $C_3$ is isomorphic to
    the Jacobian of $\widetilde{C_3 \times X_0(3)}$.
\end{lemma}
\begin{proof}
    A quick calculation shows that for every point $P$ of $X(1)$,
    the ramification indices of
    $X_0(3) \rightarrow X(1)$ above $P$,
    all divide the ramification indices of
    $C_3 \rightarrow X(1)$ above $P$, which implies
    \[ \xymatrix{ \widetilde{C_3 \times X_0(3)} \ar[r] & C_3} \]
    is unramified. Therefore, by Riemann-Hurwitz's theorem
    we get that $\widetilde{C_3\times X_0(3)} \rightarrow X(1)$
    is a genus one curve.

    To show that this map is the multiplication by the $2$ map
    on the Jacobians, note that this a geometric statement,
    and it suffices to prove it for a particular twist of $C_3$.
    Specifically, we show that
    $\widetilde{C \times X_0(3)} \rightarrow C$ over $\Q$
    induces the multiplication by $2$ map on the Jacobians,
    where $C: x^3+y^3+z^3=0$.
    The induced map on the Jacobians
    \[ \xymatrix{ \jac(C) \ar[r] & \jac(\widetilde{C \times X_0(3)})  } \]
    is an isogeny of degree $4$, and is defined over $\Q$.
    Notice that $\jac(C)$ is the elliptic curve given by
    $Y^2=X^3-2^4\cdot 3^3.$
    This elliptic curve has no rational $2$ isogeny, therefore
    the only possibility is for the above map to be multiplication by
    $2$, which is the desired result.
\end{proof}

The following proposition tells us that for most problems, $0 \not
\in \overline{\mathcal{T}_{3,p}}$ for infinitely many primes $p$.
\begin{proposition} \label{prop c3 heuristics}
    Let $J$ be the Jacobian of $C_3/\Q$ and let $p \equiv 1 \pmod 3$ be
    a prime. Assume that $J$ has good reduction at $p$.
    Then $2 | a_p(J)$ if and only if
    $0 \in \overline{\mathcal{T}_{3,p}}$.
    Specifically if $4abc$ is not a perfect cube, then
    $0 \not \in \overline{\mathcal{T}_{3,p}}$ infinitely often.
\end{proposition}
\begin{proof}
    By Lemma \ref{lem t3p}, we need to show that
    $2 | a_p(J)$ if and only if
    $\left(C_3 \times X_0(3)\right)(\F_p) \rightarrow C_3(\F_p)$
    is not surjective.
    Notice that we can replace $C_3 \times X_0(3)$ by its desingularization
    without loss of generality.

    By Lemma \ref{lem genus}, we know that $\widetilde{C_3 \times X_0(3)}$
    is a genus one curve.
    Using the Weil bound we
    get that genus one curves always have an $\F_p$
    point, and hence they are isomorphic to their Jacobians. Let
    $P \in \left(\widetilde{C_3 \times X_0(3)}\right) (\F_p)$, and let
    $Q$ be the image of
    this point in $C_3(\F_p)$.  Using $P$ and $Q$ as the origins, we get an
    explicit isomorphism between
    $C_3 \simeq J \simeq \widetilde{C_3 \times X_0(3)}$.
    Using this identification, the map
    $\widetilde{C_3 \times X_0(3)} \rightarrow C_3$ is the
    multiplication by $2$ map.

    Now if we assume that $a_p(J)$ is odd, then $J(\F_p)$ is an Abelian
    group with an odd order, therefore the multiplication by $2$ map
    is an isomorphism. In particular
    \[ \left(\widetilde{C_3 \times X_0(3)}\right)(\F_p) \rightarrow C_3(\F_p) \]
    is surjective.
    Similarly if $a_p(J)$ is even, then $J(\F_p)$ is an Abelian group with
    even order, and therefore multiplication by $2$ map is not surjective
    on the $\F_p$ points.

    Note that, when $4abc$ is not a perfect cube, the Jacobian $J$,
    given by (\ref{eqn jacobian}), has no
    nontrivial rational $2$-torsion point. In this situation, $a_p(J)$ is odd
    infinitely often, and we get that for infinitely many $p$'s, $0
    \not \in \overline{\mathcal{T}_{3,p}}$.
    This completes the proof of the proposition.
\end{proof}

In all our cases, it is easy to find an integer solution $(x,y,z)$
to (\ref{eqn twisted fermat}) with $n=3$ such that
$\rho_3^{E_{3,(x,y)}} \simeq \rho_\lambda^f$.
This shows that $a_p(f) \mod \lambda \in \overline{\mathcal{T}_{3,p}}$
for all primes $p \ndiv 3N_0$.
However, since
\[\overline{\mathcal{T}_{9,p}} \subset \overline{\mathcal{T}_{3,p}},\]
a priori it is possible that Kraus' argument can succeed for some
prime $p$ beyond our search space. This is in fact not the case,
and we show that for $p$ large enough, this containment is in fact
equality, and hence we can not get a contradiction (it is
easy to see that for $p \equiv 2 \pmod 3$ we have
$\overline{\mathcal{T}_{9,p}} = \overline{\mathcal{T}_{3,p}} =\F_3$).
\begin{proposition} \label{prop c9 heuristics}
    \begin{enumerate}
        \item If $p>216^2$ is a prime congruent to $1 \pmod 3$, then
            $\pm 1 \in \overline{\mathcal{T}_{3,p}}$ if and only if
            $\pm 1 \in \overline{\mathcal{T}_{9,p}}$.
        \item If $p>106^2$ is a prime congruent to $1 \pmod 3$,
            then $0 \in \overline{\mathcal{T}_{3,p}}$
            if and only if $0 \in \overline{\mathcal{T}_{9,p}}$.
    \end{enumerate}
\end{proposition}
\begin{proof}
    To prove the first part of the claim, by Lemma \ref{lem t3p},
    we need to show
    $C_9 \times X_0(3)$ has an $\F_p$ rational point, however this
    follows from the Weil bound
    \[ \left|\left(\widetilde{C_9 \times X_0(3)}\right)(\F_p)\right| > p+1-2g\sqrt{p},\]
    where $g$ is the genus of $C_9 \times X_0(3)$.
    To calculate this genus, note that
    $\widetilde{C_9 \times X_0(3)} \rightarrow C_9$ is unramified of degree $4$,
    and we know that the genus of $C_9$ is $(9-1)(9-2)/2=28$.
    Therefore, using Riemann-Hurwitz's theorem we get that $g=109$.
    Therefore
    \[ \left|\left(\widetilde{C_9 \times X_0(3)}\right)(\F_p)\right| > p+1-218\sqrt{p},\]
    which means for $p>218^2$, the curve
    $\widetilde{C_9 \times X_0(3)}$ will have an $\F_p$ point. This finishes
    the proof of first part of the proposition.

    To prove the second part, assume that $0 \in \overline{\mathcal{T}_{3,p}}$.
    Then, by Lemma \ref{lem t3p} the map
    $\left(\widetilde{C_3 \times X_0(3)}\right)(\F_p) \rightarrow C_3(\F_p)$
    is not surjective. Assume that
    $\left(\widetilde{C_9 \times X_0(3)}\right)(\F_p) \rightarrow C_9(\F_p)$
    is surjective. We want to show that $p<106^2$.
    Let $P \in C_9(\F_p)$. Then we can find a point
    $Q \in \left(\widetilde{C_9 \times X_0(3)}\right)(\F_p)$ which maps to $P$.
    This implies that $\phi(P)$ is in the image of
    $\left(\widetilde{C_3\times X_0(3)}\right)(\F_p) \rightarrow C_3(\F_p)$.
    However, since this map is just multiplication by $2$, and since
    it is not surjective, there are either $2$ or $4$ points that map
    to $\phi(P)$. This implies that there are either $2$ or $4$ points
    in $\left(\widetilde{C_9 \times X_0(3)}\right)(\F_p)$ that map to $P$.
    Therefore
    \[ \left| \left(\widetilde{C_9 \times X_0(3)}\right) (\F_p)\right| \geq 2|C_9(\F_p)|.\]
    However, the Weil bound tell us that
    \begin{eqnarray*}
        |C_9(\F_p)| & \geq p+1- 2\cdot 28 \sqrt{p} \\
        \left|\left(\widetilde{ C_9 \times X_0(3)}\right)(\F_p)\right| & \leq p+1 +2 \cdot 109 \sqrt{p}.
    \end{eqnarray*}
    Using these estimates we get that $p< 106^2$, as desired.
\end{proof}

\begin{remark}
    For the proof of Proposition \ref{prop c9 heuristics}, we used the
    most basic bounds for simplicity of the exposition. However, since
    the curves we are using are fairly special, much better bounds are
    known.
\end{remark}
We can now readily find all primes $p$ such that $\overline{\mathcal{T}_{9,p}}
\neq \overline{\mathcal{T}_{3,p}}$. Table \ref{table t9 neq t3} collects
this data: the column labeled $p_0$ is the set of primes
such that $0 \in \overline{\mathcal{T}_{3,p}}$ but $0 \not \in \overline{\mathcal{T}_{9,p}}$,
and the column labeled $p_1$ is the set of primes such that
$\pm 1 \not \in \overline{\mathcal{T}_{9,p}}$.
For all such primes, we can check that
$a_p(f) \in \overline{\mathcal{T}_{9,p}}$ which proves that we can
not rule out $f$ using mod $3$ level lowering.

\begin{table}
\begin{center}
\begin{tabular}{c|c|c}
$(a,b,c)$ & $p_0$ & $p_1$
\\ \hline
$(5^2,2^4,23^4)$ & $\{73, 163\}$ & $\emptyset$ \\
$(5^8,2^4,37)$ & $\{73, 307, 541\}$ & $\{37\}$ \\
$(5^7,2^4,59^7)$ & $\{37, 73, 163, 181, 199, 541\}$ & $\emptyset$ \\
$(7,2^4,47^7)$ & $\emptyset$ & $\{109\}$ \\
$(11,2^4,5^2 \cdot 17^2)$ & $\{37, 73, 307, 541\}$ & $\emptyset$
\end{tabular}
\caption{Primes $p$ where $\overline{\mathcal{T}_{3,p}} \neq \overline{\mathcal{T}_{9,p}}$.}
\label{table t9 neq t3}
\end{center}
\end{table}

\bibliographystyle{amsplain}
\bibliography{LLModPrimePowers}

\end{document}